\documentclass[a4paper,10pt,leqno]{amsart}
\usepackage[ansinew]{inputenc}
\usepackage{blowup}
\blowUp{paper={x1.10}} 

\usepackage{url}

\usepackage[colorlinks=true,linkcolor=blue,citecolor=magenta,urlcolor=red]{hyperref}

\hypersetup{pdfinfo={
  Title=A Parabolic Cylinder Function in the Riemann-Siegel Integral Formula,
  Author=Wolfgang Gabcke,
  Subject=Transformation of the Riemann-Siegel integral formula,
  Keywords={Riemann zeta function, Riemann-Siegel integral formula, parabolic cylinder function,
            functional equation of a Mordell integral}}
}

\DeclareMathOperator{\coloneqq}{%
   \mathrel{\mathop:}=\,}

\newcommand{\scfrac}[2]{%
   {\textstyle\frac{\displaystyle\raisebox{0.210ex}{\(#1\)}}{\displaystyle\raisebox{-0.710ex}{\(#2\)}}}%
}
\theoremstyle{plain}
\newtheorem{thm}{Theorem}[section]

\theoremstyle{definition}
\newtheorem{defn}{Definition}[section]

\theoremstyle{remark}

\newtheorem*{rem1}{Remark~1}
\newtheorem*{rem2}{Remark~2}

\usepackage{scrextend}

\begin{document}
%
\title[A PARABOLIC CYLINDER FUNCTION IN THE RIEMANN-SIEGEL INTEGRAL FORMULA]
      {A PARABOLIC CYLINDER FUNCTION\\IN THE RIEMANN-SIEGEL INTEGRAL FORMULA}
\author[WOLFGANG GABCKE]{WOLFGANG GABCKE}
\address{Wolfgang Gabcke, Göttingen, Germany}
\email{wolfgang@gabcke.de}

\subjclass[2010]{Primary 11M06; Secondary 33E30, 44A20}
\keywords{Riemann zeta function, Riemann-Siegel integral formula, parabolic cylinder function,
          functional equation of a Mordell integral}
\date{3rd of December 2015; 1st release: Error messages (even linguistic) and suggestions are ap\-pre\-ci\-at\-ed}

\begin{abstract}
We show that the two integrals in the Riemann-Siegel integral formula can be transformed into integral
representations that contain the parabolic cylinder function \(U(a,z)\) as kernel function.
\end{abstract}
%
%
\maketitle
%
\deffootnotemark{\textsuperscript{(\thefootnotemark)}}
\deffootnote{2em}{1.6em}{(\thefootnotemark)\enskip}
\setcounter{tocdepth}{1}
\vspace{-3mm}
\tableofcontents
\vspace{-5mm}
%
%
%
%
%
%
\section{Introduction}
\noindent
The Riemann-Siegel integral formula, first published by C.\ L.\ Siegel in 1932 (see \cite{Sie1, Sie2}), is the
symmetric integral representation
   \begin{gather*}
      \pi^{-s/2}\,\Gamma{\Bigl(\scfrac{s}{2}\Bigr)}\,\zeta{(s)}
      =\mathcal{F}(s)+\overline{\mathcal{F}(1-\overline{s})}
   \end{gather*}
of Riemann's zeta function where \(s\in\mathbb{C}\) and
   \begin{gather*}
      \mathcal{F}(s)\coloneqq\pi^{-s/2}\,\Gamma{\Bigl(\scfrac{s}{2}\Bigr)}\!\!\int\limits_{0 \swarrow 1}\!\!
      \scfrac{e^{i\pi u^2}u^{-s}}{e^{i\pi u}-e^{-i\pi u}}\,du.
   \end{gather*}
We will show that the function \(\mathcal{F}(s)\) can be written in the form
   \begin{gather}
   \label{m0015}
      \mathcal{F}(s)=2^{s/2}\,\Gamma{\Bigl(\scfrac{s}{2}\Bigr)}\,e^{-i\pi(1-s)/4}\!\!
      \int\limits_{\!\!\!\!\!\!\!{\scriptstyle-\frac{1}{2}}
      \searrow{\scriptstyle\frac{1}{2}}}\!\!\scfrac{e^{-i\pi u^2\!/2+i\pi u}}{2i\,\cos{\pi u}}\,
      U\!\left(s-1/2,\sqrt{2\pi}\,e^{i\pi/4}\,u\right)du,
   \end{gather}
i.\ e., with an integral that contains the parabolic cylinder function \(U(a,z)\) as kernel function.

In section \ref{parabolic_cylinder_function} we present a real integral representation of \(U(a,z)\) which is used
in section \ref{transformation_of_RSF}. There we will see that the integral representation \eqref{m0015} is the
consequence of the transformation of a special Mordell integral. The latter is proved in the appendix.

We will use a notation of Siegel from \cite{Sie1, Sie2} here:
\textit{Whenever there is an integration path with arrow,
the path is a straight line of slope 1 or -1 from \(\infty\) to \(\infty\) in the direction of arrow
crossing the real axis between the two points specified.}
%
%
%
%
%
%
\section{The Parabolic Cylinder Function}
\label{parabolic_cylinder_function}
\noindent
The parabolic cylinder function \(U(a,z)\) can be defined by the contour integral
   \begin{gather}
   \label{m0001}
      U(a,z)\coloneqq\scfrac{\Gamma{(1/2-a)}}{2\pi i}\,\,e^{-z^2\!/4}\!
              \int\limits_{\!\!\!C}\!e^{-w^2\!/2+zw}\,w^{a-1/2}\,dw\rlap{\(\qquad(a,\,z \in \mathbb{C})\),}
   \end{gather}
where the path \(C\) comes from \(\infty\) in the third quadrant parallel to the real axis, circles the origin once
in the positive direction and returns to \(\infty\) in the second quadrant parallel to the real axis
again.\footnote{See \cite[19.5.1]{Abr}.} \(U(a,z)\) is an entire function of both variables and a solution of the
differential equation \(y''-(z^2\!/4+a)\,y=0\). It satisfies a lot of recurrence relations such as\footnote{See
\cite[19.6.4]{Abr}.}
   \begin{gather*}
      z\,U(a,z)-U(a-1,z)+(a+1/2)\,U(a+1,z)=0.
   \end{gather*}
We will show that the two integrals in the Riemann-Siegel integral formula\footnote{To this formula see
\cite[ch.\ 7]{Edw1, Edw2}, \cite{Sie1, Sie2} or \cite[ch.\ 4, 4.1.1]{Gab}} can be transformed into integrals that
contain \(U(a,z)\) as kernel function. For this we need
%
\begin{thm}[Real integral representation]
\label{real_integral_representation}
If \(a,\,z\in\mathbb{C}\) and \(\Re{(a)}>-1/2\), we have
   \begin{gather*}
      U(a,z)=\scfrac{e^{-z^2\!/4}}{\Gamma{(1/2+a)}}\int\limits_{\!\!0}^{\,\,\infty}\!
             e^{-w^2\!/2-zw}\,w^{a-1/2}\,dw.
   \end{gather*}
\end{thm}
%
\begin{proof}
\allowdisplaybreaks{
Let \(\Re{(a)}>-1/2\) in \eqref{m0001}. Then we can put the integration path on the real axis and through the
origin giving
   \begin{gather*}
      \int\limits_{\!\!\!C}\!e^{-w^2\!/2+zw}\,w^{a-1/2}\,dw\\*
         =\int\limits_{\!\!\!-\infty}^{\,\,0}\!e^{-w^2\!/2+zw+(a-1/2)[\log{(-w)}-i\pi]}\,dw
          +\!\!\!\!\int\limits_{\!\!0}^{\,\,-\infty}\!\!\!e^{-w^2\!/2+zw+(a-1/2)[\log{(-w)}+i\pi]}\,dw\\
      \begin{aligned}
         &=\Bigl[e^{-i\pi(a-1/2)}-e^{i\pi(a-1/2)}\Bigr]
           \int\limits_{\!\!0}^{\,\,\infty}\!e^{-w^2\!/2-zw}\,w^{a-1/2}\,dw\\*
         &=2i\sin{\pi\Bigl(\scfrac{1}{2}-a\Bigr)}\!
           \int\limits_{\!\!0}^{\,\,\infty}\!e^{-w^2\!/2-zw}\,w^{a-1/2}\,dw.
      \end{aligned}
   \end{gather*}
Therefore, the theorem follows from \eqref{m0001} by using the reflection formula of the gamma function.
}
\end{proof}
%
%
%
%
%
%
%
\section{Transformation of the Riemann-Siegel Integral Formula}
\label{transformation_of_RSF}
\noindent
Setting \(\tau = 1\) and \(x=z+1/2\) in the transformation formula of the Mordell integral \eqref{m0013},
we obtain
   \begin{gather}
   \label{m0002}
      \int\limits_{0 \nearrow 1}\!\!\!\scfrac{e^{i\pi u^2+2\pi izu}}{e^{i\pi u}-e^{-i\pi u}}\,du=\!\!
      \int\limits_{0 \nwarrow 1}\!\!\scfrac{e^{-i\pi u^2+i\pi u}}{e^{-2\pi iu}-1}\,
      e^{-i\pi z^2+2\pi i\,(u-1/2)\,z}\,du.
   \end{gather}
We multiply this equation by \(z^{s-1}\) with \(s\) real \(>1\) and integrate over \(z\) along the bisecting
line of the second quadrant. This yields for the left side of \eqref{m0002}\footnote{The resulting double
integrals converge absolutely so that the integration order can be interchanged. See also theorem
\cite[4.1.6]{Gab}.}
   \begin{align*}
      \int\limits_{\!\!0}^{e^{\frac{3\pi i}{4}}\infty}\!\!\!\!\!
      \int\limits_{0 \nearrow 1}\!\!\scfrac{e^{i\pi u^2+2\pi izu}\,z^{s-1}}{e^{i\pi u}-e^{-i\pi u}}\,du\,dz
      &=\!\!\int\limits_{0 \nearrow 1}\!\!\scfrac{e^{i\pi u^2}}{e^{i\pi u}-e^{-i\pi u}}\!\!\!
      \int\limits_{\!\!0}^{e^{\frac{3\pi i}{4}}\infty}\!\!\!\!\!e^{2\pi izu}z^{s-1}\,dz\,du\\*
      &=e^{i\pi s/2}\,(2\pi)^{-s}\,\Gamma{(s)}\!\!
        \int\limits_{0 \nearrow 1}\!\!\!\scfrac{e^{i\pi u^2}u^{-s}}{e^{i\pi u}-e^{-i\pi u}}\,du
   \end{align*}
and for its right side
   \begin{gather*}
      \int\limits_{\!\!0}^{e^{\frac{3\pi i}{4}}\infty}\!\!\!
      \int\limits_{0 \nwarrow 1}\!\!\scfrac{e^{-i\pi u^2+i\pi u}}{e^{-2\pi iu}-1}\,
      e^{-i\pi z^2+2\pi i\,(u-1/2)\,z}\,z^{s-1}\,du\,dz\\*
         =\!\!\int\limits_{0 \nwarrow 1}\!\!\scfrac{e^{-i\pi u^2+i\pi u}}{e^{-2\pi iu}-1}\!\!
           \int\limits_{\!\!0}^{e^{\frac{3\pi i}{4}}\infty}\!\!\!\!\!
           e^{-i\pi z^2+2\pi i\,(u-1/2)\,z}\,z^{s-1}\,dz\,du\\*[1.0ex]
         =(2\pi)^{-s/2}\,\Gamma{(s)}\,e^{3i\pi s/4+i\pi/8}\!\!\int\limits_{0 \nwarrow 1}\!\!
           \scfrac{e^{-i\pi u^2\!/2+i\pi u/2}}{e^{-2\pi iu}-1}\,U\!\left(s-1/2,\sqrt{2\pi}\,
           e^{i\pi/4}\,(u-1/2)\right)du,
   \end{gather*}
where we have used the formula
   \begin{gather*}
      \int\limits_{\!\!0}^{e^{\frac{3\pi i}{4}}\infty}\!\!\!\!\!
      e^{-i\pi z^2+2\pi i\,(u-1/2)\,z}\,z^{s-1}\,dz\\*
      =(2\pi)^{-s/2}\,\Gamma{(s)}\,e^{3i\pi s/4+i\pi (u-1/2)^2\!/2}\,
      U\!\left(s-1/2,\sqrt{2\pi}\,e^{i\pi/4}\,(u-1/2)\right),
   \end{gather*}
which follows from theorem \ref{real_integral_representation} if we substitute there
\(we^{3\pi i/4}\!/\!\sqrt{2\pi}\) for \(z\) in the integral. With these transformations from \eqref{m0002}
   \begin{gather*}
      \int\limits_{0 \nearrow 1}\!\!\!\scfrac{e^{i\pi u^2}u^{-s}}{e^{i\pi u}-e^{-i\pi u}}\,du\\
      =(2\pi)^{s/2}\,e^{i\pi s/4+i\pi/8}\!\!\int\limits_{0 \nwarrow 1}\!\!
       \scfrac{e^{-i\pi u^2\!/2+i\pi u/2}}{e^{-2\pi iu}-1}\,U\!\left(s-1/2,\sqrt{2\pi}\,
       e^{i\pi/4}\,(u-1/2)\right)du
   \end{gather*}
follows. Finally, if we substitute \(u+1/2\) for \(u\) in the last integral, we have
   \begin{gather*}
      \int\limits_{0 \nearrow 1}\!\!\!\scfrac{e^{i\pi u^2}u^{-s}}{e^{i\pi u}-e^{-i\pi u}}\,du\\
      =(2\pi)^{s/2}\,e^{i\pi(s+1)/4}\!\!\int\limits_{\!\!\!\!\!\!\!{\scriptstyle-\frac{1}{2}}
      \searrow{\scriptstyle\frac{1}{2}}}\!\!\scfrac{e^{-i\pi u^2\!/2}}{e^{-2\pi iu}+1}\,
      U\!\left(s-1/2,\sqrt{2\pi}\,e^{i\pi/4}\,u\right)du
   \end{gather*}
and so
   \begin{gather}
      \notag \mathcal{F}(s)\coloneqq\pi^{-s/2}\,\Gamma{\Bigl(\scfrac{s}{2}\Bigr)}\!\!
      \int\limits_{0 \swarrow 1}\!\!
      \scfrac{e^{i\pi u^2}u^{-s}}{e^{i\pi u}-e^{-i\pi u}}\,du\\[-1.5ex]
   \label{m0005}\phantom{a}\\[-1.5ex]
      \notag =2^{s/2}\,\Gamma{\Bigl(\scfrac{s}{2}\Bigr)}\,e^{-i\pi(1-s)/4}\!\!
      \int\limits_{\!\!\!\!\!\!\!{\scriptstyle-\frac{1}{2}}
      \searrow{\scriptstyle\frac{1}{2}}}\!\!\scfrac{e^{-i\pi u^2\!/2+i\pi u}}{2i\,\cos{\pi u}}\,
      U\!\left(s-1/2,\sqrt{2\pi}\,e^{i\pi/4}\,u\right)\,du
   \end{gather}
as well as
   \begin{gather*}
      \overline{\mathcal{F}(1-\overline{s})}=\pi^{-(1-s)/2}\,\Gamma{\Bigl(\scfrac{1-s}{2}\Bigr)}\!\!
      \int\limits_{0 \searrow 1}\!\!\scfrac{e^{-i\pi u^2}u^{s-1}}{e^{i\pi u}-e^{-i\pi u}}\,du\\
      =2^{(1-s)/2}\,\Gamma{\Bigl(\scfrac{1-s}{2}\Bigr)}\,e^{i\pi s/4}\!\!
      \int\limits_{\!\!\!\!\!\!\!{\scriptstyle-\frac{1}{2}}
      \swarrow{\scriptstyle\frac{1}{2}}}\!\!\scfrac{e^{i\pi u^2\!/2-i\pi u}}{2i\,\cos{\pi u}}\,
      U\!\left(1/2-s,\sqrt{2\pi}\,e^{-i\pi/4}\,u\right)du,
   \end{gather*}
now valid for all \(s\) by analytic continuation. So the Riemann-Siegel integral formula
   \begin{gather*}
      \pi^{-s/2}\,\Gamma{\Bigl(\scfrac{s}{2}\Bigr)}\,\zeta{(s)}
      =\mathcal{F}(s)+\overline{\mathcal{F}(1-\overline{s})}
   \end{gather*}
yields an integral representation that contains the parabolic cylinder function \(U\) as integral
kernel.

It should be noted that the cosine function on the right side of \eqref{m0005} is essential for the integral to
converge in the second quadrant.
%
%
%
%
%
%
\section{Appendix}
\label{appendix}
\noindent
We prove some theorems used above that are hard to find in the literature.
\begin{defn}[Mordell integral]
\label{def_mordell_integral}
For \(x,\,\tau \in \mathbb{C}\) and \(\Re{(\tau)>0}\) let
   \begin{gather}
   \label{m0010}
      \Phi{(x,\tau)}\coloneqq\!\!\int\limits_{0 \nearrow 1}\!\!
      \scfrac{e^{i\pi \tau u^2+2\pi ixu}}{e^{2\pi iu}-1}\,du
   \end{gather}
be a special Mordell integral\footnote{For definition and properties of Mordell integrals and their connection to
theta functions see \cite{Mor} and for some applications \cite[ch.\ 4.1]{Gab}.}.
The integration path is a straight line of slope \(1\) from the lower to the upper halfplane crossing the real
axis between \(0\) and \(1\).
\end{defn}
\noindent
\(\Phi{(x,\tau)}\) is an entire function of \(x\) because the integral converges regardless of the value of
\(x\).

If \(\tau\) is a positive rational number, a representation of \(\Phi{(x,\tau)}\) using only exponential functions
is given by
%
\begin{thm}
\label{phi_with_rational_tau}
Let \(m,\,n\) be natural numbers and \(\tau=m/n\). Then
   \begin{gather}
   \label{m0011}
      \Phi{\Bigl(x,\scfrac{m}{n}\Bigr)}=
      \scfrac{\displaystyle\sum\limits_{k=1}^{n}e^{i\pi{\textstyle\frac{m}{n}}k^2+2k\pi ix}-
      \sqrt{\scfrac{n}{m}}\,e^{i\pi\left({\textstyle\frac{1}{4}}-{\textstyle\frac{n}{m}}x^2\right)}
      \displaystyle\sum\limits_{k=1}^{m}
      e^{-i\pi{\textstyle\frac{n}{m}}k^2+2k\pi i{\textstyle\frac{n}{m}}x}}{\displaystyle e^{i\pi n(2x+m)}-1}.
      \;\raisebox{1.2ex}{\footnotemark}
   \end{gather}
\footnotetext{Since \(\Phi{(x,\tau)}\) is an entire function of \(x\), every zero of the denominator must be also
a zero of the numerator. From that we obtain interesting relations between exponential sums and an elegant proof
of the quadratic reciprocity law (see \cite[ch.\ V]{Cha1} or \cite[p.\ 35 ff.]{Cha2}).}
\end{thm}
%
\noindent
For a proof of this fundamental theorem see \cite[ch.\(\,\)V]{Cha1} or \cite[p.\ 35 ff.]{Cha2}.
%
\begin{thm}[Functional equation of the Mordell integral]
\label{functional_equation_mordell_integral}
Let \(x,\tau \in \mathbb{C}\) and \(\Re{(\tau)>0}\). Then the functional equation
   \begin{gather}
   \label{m0012}
      \Phi{(x,\tau)}=-\scfrac{e^{i\pi (1/4-x^2\!/\tau)}}{\sqrt{\tau}}\,
      \overline{\Phi{\Bigl(\!-\scfrac{\overline{x}}{\overline{\tau}},\scfrac{1}{\overline{\tau}}\Bigr)}}\\
   \intertext{and the integral transformation formula}
   \label{m0013}
      \int\limits_{0 \nearrow 1}\!\!\scfrac{e^{i\pi \tau u^2+2\pi ixu}}{e^{2\pi iu}-1}\,du
      =\scfrac{e^{i\pi (1/4-x^2\!/\tau)}}{\sqrt{\tau}}\!\!
      \int\limits_{0 \nwarrow 1}\!\!\scfrac{e^{-i\pi u^2\!/\tau+2\pi ixu/\tau}}{e^{-2\pi iu}-1}\,du
   \end{gather}
are valid.
\end{thm}
%
\begin{proof}
We exchange \(n\) and \(m\) in \eqref{m0011}. Then we replace \(x\) by \(-n\overline{x}/m\) and obtain
   \begin{gather*}
      \Phi{\Bigl(\!-\scfrac{n\overline{x}}{m},\scfrac{n}{m}\Bigr)}=
      \scfrac{\displaystyle\sum\limits_{k=1}^{m}
      e^{i\pi{\textstyle\frac{n}{m}}k^2-2k\pi i{\textstyle\frac{n}{m}}\overline{x}}-
      \sqrt{\scfrac{m}{n}}\,e^{i\pi\left({\textstyle\frac{1}{4}}-{\textstyle\frac{n}{m}}\overline{x}^2\right)}
      \displaystyle\sum\limits_{k=1}^{n}
      e^{-i\pi{\textstyle\frac{m}{n}}k^2-2k\pi i\overline{x}}}{\displaystyle e^{i\pi m(-2n\overline{x}/m+n)}-1}.
   \end{gather*}
Transition to the complex conjugated value yields
   \begin{gather*}
      \overline{\Phi{\Bigl(\!-\scfrac{n\overline{x}}{m},\scfrac{n}{m}\Bigr)}}=
      \scfrac{\displaystyle\sum\limits_{k=1}^{m}
      e^{-i\pi{\textstyle\frac{n}{m}}k^2+2k\pi i{\textstyle\frac{n}{m}}x}-
      \sqrt{\scfrac{m}{n}}\,e^{-i\pi\left({\textstyle\frac{1}{4}}-{\textstyle\frac{n}{m}}x^2\right)}
      \displaystyle\sum\limits_{k=1}^{n}
      e^{i\pi{\textstyle\frac{m}{n}}k^2+2k\pi ix}}{\displaystyle e^{i\pi n(2x-m)}-1}.
   \end{gather*}
We multiply this equation by \(-\sqrt{n/m}\,e^{i\pi(1/4-nx^2\!/m)}\) and regain the right side of \eqref{m0011} on
the right side here since \(e^{i\pi n(2x-m)}=e^{i\pi n(2x+m)}\). Thus we have proved
   \begin{gather}
   \label{m0014}
      \Phi{(x,\tau)}=-\scfrac{e^{i\pi (1/4-x^2\!/\tau)}}{\sqrt{\tau}}\,
      \overline{\Phi{\Bigl(\!-\scfrac{\overline{x}}{\tau},\scfrac{1}{\tau}\Bigr)}}
   \end{gather}
if \(\tau=m/n\) is any positive rational number and \(x\in\mathbb{C}\). But since the rational numbers are dense in
\(\mathbb{R}\), this equation must hold for all positive real numbers. Unfortunately, we cannot generalise this
result to complex \(\tau\) directly because the complex conjugation is not a holomorphic operation.
We use another way and replace \(\tau\) by \(1/\tau\) and \(x\) by \(-\overline{x}/\tau\) in definition
\eqref{m0010} giving
   \begin{gather*}
      \Phi{\Bigl(\!-\scfrac{\overline{x}}{\tau},\scfrac{1}{\tau}\Bigr)}
      =\!\!\int\limits_{0 \nearrow 1}\!\!\scfrac{e^{i\pi u^2\!/\tau-2\pi i\overline{x}u/\tau}}{e^{2\pi iu}-1}\,du
   \end{gather*}
and therefore,
   \begin{gather*}
      \overline{\Phi{\Bigl(\!-\scfrac{\overline{x}}{\tau},\scfrac{1}{\tau}\Bigr)}}
      =-\!\!\!\int\limits_{0 \nwarrow 1}\!\!\scfrac{e^{-i\pi u^2\!/\tau+2\pi ixu/\tau}}{e^{-2\pi iu}-1}\,du.
   \end{gather*}
We put this in \eqref{m0014} and under use of \eqref{m0010} we obtain the integral transformation formula
\eqref{m0013} valid for positive real \(\tau\) and \(x\in\mathbb{C}\). Obviously, both sides of \eqref{m0013} are
analytic functions of \(\tau\) if \(\Re{(\tau)}>0\). So this equation holds for these \(\tau\) by analytic
continuation as required. Finally, the functional equation \eqref{m0012} follows from \eqref{m0013} in combination
with \eqref{m0010} because of
   \begin{gather*}
      \int\limits_{0 \nwarrow 1}\!\!\scfrac{e^{-i\pi u^2\!/\tau+2\pi ixu/\tau}}{e^{-2\pi iu}-1}\,du
      =\overline{\!\!\int\limits_{0 \swarrow 1}\!\!
       \scfrac{e^{i\pi u^2\!/\overline{\tau}-2\pi i\overline{x}u/\overline{\tau}}}{e^{2\pi iu}-1}\,du}
      =-\overline{\Phi{\Bigl(\!-\scfrac{\overline{x}}{\overline{\tau}},\scfrac{1}{\overline{\tau}}\Bigr)}}.
   \end{gather*}
\end{proof}
%
\begin{rem1}\textit{\(\!\!\!\)
The functional equation \eqref{m0012} differs only marginally from that of the Jacobian theta function
\(\sum_{n=-\infty}^{+\infty}e^{i\pi\tau n^2+2n\pi ix}\:\,(\Im{(\tau)>0})\). Indeed, Mordell has proved the
strong connection between an integral like \eqref{m0010} and the theta functions in \cite{Mor}.}
\end{rem1}
\begin{rem2}\textit{\(\!\!\!\)
The integral transformation \eqref{m0013} can be used in the case of small \(|\tau|\) to con\-vert a slowly
convergent integral into a rapidly convergent one.}
\end{rem2}
%
%
%
%
%
%


\begin{thebibliography}{1}
\bibitem{Abr} \textsc{M.\(\,\,\)Abramowitz, I.\(\,\,\)A.\(\,\,\)Stegun},
\textit{Handbook of Mathematical Functions}, National Bureau of Standards,
Applied Mathematics Series, \textbf{55}, Tenth Printing, 1972,
\url{http://www.cs.bham.ac.uk/~aps/research/projects/as/book.php}
%
\medskip
\bibitem{Cha1} \textsc{K.\(\,\,\)Chandrasekharan},
\textit{Ein\-füh\-rung in die ana\-ly\-ti\-sche Zah\-len\-the\-orie}, Lecture Notes in Mathematics \textbf{29},
Springer Verlag, Berlin, Heidelberg, New York, 1966
\bibitem{Cha2} \raisebox{1mm}{\line(1,0){30}},
\textit{Introduction to Analytic Number Theory}, Springer Verlag, Berlin, Heidelberg, New York, 1968
%
\medskip
\bibitem{Edw1} \textsc{H.\(\,\,\)M.\(\,\,\)Edwards}, \textit{Riemann's Zeta Function}, Academic Press,
New York, 1974
%
\bibitem{Edw2} \raisebox{1mm}{\line(1,0){30}}, \raisebox{1mm}{\line(1,0){30}}, reprint,
Dover Publications, New York, 2001
%
\medskip
\bibitem{Gab} \textsc{W.\(\,\,\)Gabcke},
\textit{Neue Herleitung und explizite Restabschätzung der
Riemann-Siegel-Formel}, Ph.D.\ thesis, Georg-August-Universität zu Göttingen, 1979,
\url{http://hdl.handle.net/11858/00-1735-0000-0022-6013-8}
%
\medskip
\bibitem{Mor} \textsc{L.\(\,\,\)J.\(\,\,\)Mordell}, \textit{The definite integral
\(\int\limits_{\!-\infty}^{\infty}\scfrac{e^{ax^2+bx}}{e^{cx}+d}\,dx\) and the analytic theory of numbers},
Acta Mathematica \textbf{61}, 1933, 323--360,
\url{http://www.springerlink.com/content/j1474815272w103h/fulltext.pdf}
%
\medskip
\bibitem{Sie1} \textsc{C.\(\,\,\)L.\(\,\,\)Siegel},
\textit{Über Riemanns Nach\-laß zur analytischen Zahlentheorie}, Quellenstudien zur Geschichte der Mathematik,
Astronomie und Physik, \textbf{Abt.\ B, Studien 2}, 1932, 45--80
%
\bibitem{Sie2} \raisebox{1mm}{\line(1,0){30}},
\textit{Gesammelte Abhandlungen}, \textbf{1}, 275--310,
Springer Verlag, Berlin, Heidelberg, New York, 1966
\end{thebibliography}
\end{document}